\documentclass[11pt]{article}

\usepackage[margin = 1.1in]{geometry}

\usepackage{amssymb,amsmath,amsthm,hyperref,psfrag,epsfig}
\usepackage{color}
\hypersetup{colorlinks=true}

\newtheorem{theorem}{Theorem}
\newtheorem{lemma}[theorem]{Lemma}
\newtheorem{corollary}[theorem]{Corollary}

\newtheorem{remark}{Remark}
\newtheorem{definition}{Definition}

\newtheorem{observation}{Observation}

\newcommand{\bD}{\boldsymbol{D}}
\newcommand{\bc}{\boldsymbol{c}}

\newcommand{\bM}{\boldsymbol{M}}
\newcommand{\bN}{\boldsymbol{N}}
\newcommand{\bS}{\boldsymbol{S}}

\newcommand{\cF}{\mathcal{F}}

\newcommand{\orig}{\boldsymbol{o}}

\newcommand{\beq}{\begin{eqnarray}}
\newcommand{\eeq}{\end{eqnarray}}
\newcommand{\beqn}{\begin{equation}}
\newcommand{\eeqn}{\end{equation}}

\newcommand{\Z}{\mathbb{Z}}

\newcommand{\Rp}{\mathbb{R}_+}

\renewcommand{\hat}{\widehat}

\newcommand{\integers}{\mathbb{Z}}
\newcommand{\T}{\mathbb{T}}
\newcommand{\Tfour}{T_4}
\newcommand{\even}{\mathrm{even}}
\newcommand{\odd}{\mathrm{odd}}
\newcommand{\eps}{\epsilon}
\newcommand{\Tsaw}{T_{\mathrm{saw}}}
\newcommand{\dist}{\mathrm{dist}}

\newcommand{\bfrho}{\boldsymbol{\rho}}
\newcommand{\bfeta}{\boldsymbol{\eta}}

\newcommand{\FM}{\mathcal F_{\leq\bM}}
\newcommand{\Ball}{\mathrm{B}}

\newcommand{\bound}{{2.3882}}
\newcommand{\boundF}{{2.1625}}
\newcommand{\simplerstarcond}{{\star}}
\newcommand{\starcond}{\star\star}

\renewcommand{\phi}{\varphi}

\begin{document}

\title{
Improved Mixing Condition on the Grid for Counting and Sampling Independent Sets}

\author{
Ricardo Restrepo\thanks{School of Mathematics, Georgia Institute of
Technology, Atlanta GA 30332. Email:
\{restrepo,tetali\}@math.gatech.edu. Research supported in part by
NSF grant CCF-0910584.} \and Jinwoo Shin\thanks{School of Computer
Science, Georgia Institute of Technology, Atlanta GA 30332. Email:
\{jshin72,ljyang,vigoda\}@cc.gatech.edu. Research supported in part
by NSF grants CCF-0830298 and CCF-0910584.
Jinwoo Shin was supported by the Algorithms and Randomness Center
at Georgia Tech.} \and Prasad
Tetali$^\star$ \and Eric Vigoda$^\dag$ \and Linji Yang$^\dag$ }

\maketitle

\thispagestyle{empty}

\begin{abstract}
The hard-core model has received much attention in the past couple
of decades as a lattice gas model with hard constraints in
statistical physics, a multicast model of calls in communication
networks, and as a weighted independent set problem in
combinatorics, probability and theoretical computer science.

In this model, each independent set $I$ in a graph $G$ is weighted
proportionally to $\lambda^{|I|}$, for a positive  real parameter
$\lambda$.  For large $\lambda$, computing the partition function
(namely, the normalizing constant which makes the weighting a
probability distribution on a finite graph) on graphs of maximum
degree $\Delta\ge 3$,  is a  well known computationally challenging
problem.  More concretely, let $\lambda_c(\T_\Delta)$ denote  the
critical value for the so-called uniqueness  threshold of the
hard-core model on the infinite $\Delta$-regular tree;  recent
breakthrough results of Dror Weitz (2006) and Allan Sly (2010) have
identified $\lambda_c(\T_\Delta)$ as a threshold where the hardness
of estimating the above partition function undergoes a computational
transition.

We focus on the well-studied particular case of the square lattice
$\integers^2$, and provide a new  lower bound for the uniqueness
threshold, in particular taking it well above $\lambda_c(\T_4)$.
Our technique refines and builds on the tree of self-avoiding walks
approach of Weitz, resulting in a new technical sufficient criterion
(of wider applicability) for establishing strong spatial mixing (and
hence uniqueness) for the hard-core model. Our new criterion
achieves better bounds on strong spatial mixing when the graph has
extra structure, improving upon what can be achieved by just using
the maximum degree. Applying our technique to $\integers^2$ we prove
that strong spatial mixing holds for all $\lambda<\bound$, improving
upon the work of Weitz that held for $\lambda<27/16=1.6875$. Our
results imply a fully-polynomial {\em deterministic} approximation
algorithm for estimating the partition function, as well as rapid
mixing of the associated Glauber dynamics to sample from the
hard-core distribution.
\end{abstract}

\newpage

\section{Introduction}

In this paper we study phase transitions for sampling weighted
independent sets (weighted by an activity $\lambda>0$) of the
2-dimensional integer lattice $\integers^2$. In statistical physics
terminology, we study the hard-core lattice gas model
(\cite{BergSteif,GauntFisher}), which is a simple model of a gas
whose particles have non-negligible size (thus preventing them from
occupying neighboring sites), with activity $\lambda \in \Rp$
corresponding to the so-called fugacity of the gas. More formally,
for a finite graph $G=(V,E)$, let $\Omega=\Omega(G)$ denote the set
of independent sets of $G$. Given an independent set
$\sigma\in\Omega$, its weight is defined as $w(\sigma) =
\lambda^{|\sigma|}$ and $v\in V$ is said to be occupied {under
$\sigma$} if $v\in \sigma$. The associated Gibbs (or Boltzmann)
distribution $\mu=\mu_{G,\lambda}$ is defined on $\Omega$ as
$\mu(\sigma) = w(\sigma)/Z$, where
$Z=Z(G,\lambda)=\sum_{\eta\in\Omega} w(\eta)$ is commonly referred
to as the partition function.

Recall that Valiant \cite{Valiant} showed that {\em exactly
computing} the number of independent sets is \#P-complete, even when
restricted to 3-regular graphs (see Greenhill \cite{Greenhill}).
 Hence, we focus our attention on approximation algorithms for
 estimating the number, or more generally, the partition function.
It is well known \cite{JVV} that the problem of approximating the
partition function $Z$ and that of sampling from a distribution that
is close to the Gibbs distribution $\mu$, are polynomial-time
reducible to each other (see also \cite{SVV}).

The fundamental notion of a phase transition for a statistical
mechanics model on an infinite graph addresses the critical point at
which the model starts to exhibit a certain long-range dependence,
as a system parameter is varied.  In particular, the so-called
critical inverse temperature $\beta_c$ for the Ising or the Potts
model, and the critical activity $\lambda_c$ for the hard-core
lattice gas model, are prime examples where the system undergoes a
transition from uniqueness to multiplicity of the infinite-volume
Gibbs measures.

Phase transition in the hard-core model is also intimately related
to the computational complexity of estimating the partition function
$Z$. Recently, a remarkable connection was established between the
computational complexity of approximating the partition function for
graphs of maximum degree $\Delta$ and the phase transition
$\lambda_c(\T_\Delta)$ for the infinite regular tree $\T_\Delta$ of
degree $\Delta$. On the positive side, Weitz \cite{Weitz} showed a
deterministic fully-polynomial time approximation algorithm (FPAS)
for approximating the partition function for any graph with maximum
degree $\Delta$, when $\lambda<\lambda_c(\T_\Delta)$ and $\Delta$ is
constant.  On the other side, Sly \cite{Sly} recently showed that
for every $\Delta\geq 3$, it is NP-hard (unless NP=RP) to
approximate the partition function for graphs of maximum degree
$\Delta$, when
$\lambda_c(\T_\Delta)<\lambda<\lambda_c(\T_\Delta)+\eps_\Delta$, for
some function $\eps_\Delta>0$. More recently, Galanis et al.
\cite{Galanis} improved the range of $\lambda$ in Sly's
inapproximability result, extending it to all
$\lambda>\lambda_c(\T_\Delta)$ for the cases $\Delta=3$ and
$\Delta\geq 6$.

\subsection{Prior history and current work}
Our work builds upon Weitz's work to get improved results for
specific graphs of interest. We focus our attention on what is
arguably the simplest, not yet well-understood, case of interest
namely the square grid, or the 2-dimensional integer lattice
$\integers^2$. Empirical evidence suggests that the critical point
$\lambda_c(\integers^2)\approx 3.796$ \cite{GauntFisher,BET,Racz},
but rigorous results are significantly far from this conjectured
point. The possibility of there being multiple such $\lambda_c$ is
not ruled out, although no one believes that this is the case.

From below, van den Berg and Steif \cite{BergSteif}
used a disagreement percolation argument to
prove $\lambda_c(\integers^2)>\frac{p_c}{1-p_c}$ where
$p_c$ is the critical probability for site percolation on $\integers^2$.
Applying the best known lower bound on $p_c > 0.556$ for $\integers^2$
by van den Berg and Ermakov \cite{BE} implies
$\lambda_c(\integers^2)>1.252\dots$.
Prior to that work, an alternative approach
aimed at establishing the Dobrushin-Shlosman criterion \cite{DS},
yielded, via computer-assisted proofs,
$\lambda_c(\integers^2)>1.185$ by
Radulescu and Styer \cite{RS}, and
$\lambda_c(\integers^2)>1.508$ by Radulescu \cite{Rad}.

These results were improved upon by
Weitz \cite{Weitz} who showed that $\lambda_c(\integers^2) \geq
\lambda_c(\T_4) = 27/16=1.6875$, where $\T_\Delta$ is the infinite,
complete, regular tree of degree $\Delta$. For the upper bound, a
classical Peierls' type argument implies
$\lambda_c(\integers^2)=O(1)$ \cite{Dobrushin}.  (A related result
of Randall \cite{Randall} showing slow mixing of the Glauber
dynamics for $\lambda>8.066$ gives hope for a better upper bound on
$\lambda_c(\integers^2)$.) The regular tree $\T_\Delta$ is one of
the only examples (that we know of)
 where the critical point is
known exactly, and in this case, Kelly \cite{Kelly} showed that
$\lambda_c(\T_\Delta) = (\Delta-1)^{\Delta-1}/(\Delta-2)^\Delta$.

In this work we present a new general approach which, for the case
of the hard-core model on $\integers^2$, improves the lower bound to
$\lambda_c(\integers^2)> \bound$. There are various algorithmic
implications for finite subgraphs of the $\integers^2$ when
$\lambda<\bound$. Our results imply that Weitz's deterministic FPAS
is also valid on subgraphs of $\integers^2$ for the same range of
$\lambda$. Thanks to the existing literature on general spin systems
(\cite{MO1,MO2,Cesi,DSVW}), our results also imply that the Glauber
dynamics has $O(n\log{n})$ mixing time for any finite subregion
$G=(V,E)$ of $\integers^2$ when $\lambda<\bound$, where $n=|V|$.
Recall that the Glauber dynamics is a simple Markov chain that
updates the configuration at a randomly chosen vertex in each step,
see \cite{LPW} for an introduction to the Glauber dynamics. The
stationary distribution of this chain is the Gibbs distribution.
Hence, it is of interest as an algorithmic technique to randomly
sample from the Gibbs distribution, and also as a model of how
physical systems reach equilibrium.  The mixing time is the number
of steps (from the worst initial configuration) until the
distribution is guaranteed to be within variation distance $\leq
1/4$ of the stationary distribution.

As in Weitz's work, our approach can be used for other 2-spin
systems, such as the Ising model.  This is discussed in Section
\ref{sec:Ising}. As will be evident from the following high-level
idea of our approach, it can be applied to other graphs of interest.
Our work also provides an arguably simpler way to derive the main
technical result of Weitz showing that any graph with maximum degree
$\Delta$ has strong spatial mixing (SSM) when
$\lambda<\lambda_c(\T_\Delta)$.

To underline the difficulty in estimating bounds on $\lambda_c$, we
remark  that the {\em existence} of a (unique) critical activity
$\lambda_c$ remains conjectural and an open problem for
$\mathbb{Z}^d$, for $d\ge 2$. In contrast, for the Ising model, the
critical inverse temperature $\beta_c(\mathbb{Z}^2)$ has been known
since 1944 \cite{Ons44}; interestingly, the corresponding critical
point for the  $q$-state Potts model (for $q\ge 2$) has only
recently been established (by Beffara and Duminil-Copin \cite{BD10})
to be $\beta_c(q) = \log(1+\sqrt{q})$, settling a long-standing open
problem. The lack of monotonicity in $\lambda$ in the hard-core
model poses a serious challenge in establishing such a sharp result
for this model.  In fact, Brightwell et al.
\cite{brightwell1999nonmonotonic} showed that in general such a
monotonicity need not hold, by providing an example with a
non-regular tree.

\section{Technical Preliminaries and Proof Approach}

Before presenting our approach, it is useful to review briefly the
uniqueness/non-uniqueness phase transition, and introduce associated
notions of decay of spatial correlation, known as weak and strong
spatial mixing properties.  Much of the below discussion is
simplified for the case of the hard-core model on $\integers^2$,
wherein one utilizes certain induced monotonicity (given by the
bipartite property) in the model and the amenability of the graph.

\subsection{Uniqueness, Weak and Strong Spatial Mixing}
\label{sec:mixing-definitions}

Let $B_L$ denote the finite graph corresponding to a box of
side-length $2L+1$ centered around the origin in $\integers^2$.
Thus, $B_L=(V,E)$, where $V={(i,j)\in \mathbb{Z}^2:-2L-1\leq i,j\leq 2L+1}$
with  edges between pairs of vertices at $L_1$ distance (or
Manhattan distance) equal to one. Since this is a bipartite graph,
we may fix one such partition $V=\even\cup\odd$ -- for example, it
is standard to consider the set of  vertices at an even distance
from the origin as the even set. The boundary of $B_L$ are those
vertices $v=(v_1,v_2)\in V$ where $|v_i| = 2L+1$ for $i=1$ or $i=2$.
The hard-core model on bipartite graphs is a monotone system (e.g.,
see \cite{DSVW}), which for the current discussion implies that we
only have to consider two assignments to the boundary: all even
vertices or all odd vertices on the boundary are occupied. Let
$\alpha_{L,r}^{\even}$ ($\alpha_{L,r}^{\odd}$) denote the marginal
probability that the origin $r$ is unoccupied given the even (odd,
respectively) boundary. 
Then to establish uniqueness of the Gibbs measures, we need that:
\[ \lim_{L\rightarrow\infty} \vert \alpha_{L,r}^{\even} - \alpha_{L,r}^{\odd}\vert = 0.
\]
We are interested in the critical point $\lambda_c$ for the
transition between uniqueness and non-uniqueness.  A standard way to
establish uniqueness is by proving one of the  spatial mixing
properties introduced next.

Let $G=(V,E)$ be a (finite) graph.  For $S\subset V$, a
configuration $\bfrho$ on $S$ specifies a subset of $S$ as occupied
and the remainder as unoccupied. Let $\mu^{\bfrho}=\mu^{\bfrho}_G$
denote the Gibbs distribution conditional on configuration $\bfrho$
to $S$. For $v\in V$, let $\alpha^{\bfrho}_v=\alpha_{G,v}^{\bfrho}$
denote the marginal probability that $v$ is unoccupied in
$\mu^{\bfrho}$.

The first spatial mixing property is {\em Weak Spatial Mixing
(WSM)}. Here we consider a pair of boundary configurations on a
subset $S$ and consider the ``influence'' on the marginal
probability that a vertex $v$ is unoccupied.   WSM says that the
influence on $v$ decays exponentially in the distance of $S$ from
$v$.

\begin{definition}[Weak Spatial Mixing]\label{def:wsm}
For the hard-core model at activity $\lambda$, for finite graph
$G=(V,E)$, {\em WSM} holds with rate $\gamma\in(0,1)$ if for every
$v\in V$, every $S\subset V$, and every two configurations
$\bfrho,\bfeta$ on $S$,
$$\left|\alpha^{\bfrho}_v-\alpha^{\bfeta}_v\right|
~\leq~\gamma^{\dist(v,T)}$$ where $\dist(v,S)$ is the graph distance
(i.e., length of the shortest path) between $v$ and (the nearest
point in) the subset $S$.
\end{definition}

The second property of interest is {\em Strong Spatial Mixing
(SSM)}.  The intuition is that if a pair of boundary configurations
on a subset $S$ agree at some vertices in $S$ then those vertices
``encourage'' $v$ to agree. Therefore, SSM says that the influence
on $v$ decays exponentially in the distance of $v$ from the subset
of vertices where the pair of configurations differ.

\begin{definition}[Strong Spatial Mixing]
For the hard-core model at activity $\lambda$, for finite graph
$G=(V,E)$, {\em SSM} holds with rate $\gamma\in(0,1)$ if
 for every $v\in V$, every $S\subset V$, every $S'\subset S$,
and every two configurations $\bfrho,\bfeta$ on $S$ where
$\bfrho(S\setminus S')=\bfeta(S\setminus S')$,
$$\left|\alpha^{\bfrho}_v-\alpha^{\bfeta}_v\right|
~\leq~\gamma^{\dist(v,S')}.$$
\end{definition}

Note that since $\dist(v,T) \le \dist(v, T\setminus S)$, SSM implies
WSM for the same rate. Moreover, it is a standard fact that such an
exponential decay in finite boxes (say), in $\Z^d$,   implies
uniqueness of the corresponding infinite volume Gibbs measure on
$\Z^d$, see Georgii \cite{Georgii} for an introduction to the theory
of infinite-volume Gibbs measures. We can specialize the above
notions of WSM and SSM to a particular vertex $v$, in which case we
say that WSM or SSM holds \emph{at} $v$. If the graph is a rooted
tree, we will always assume that the notions of WSM and SSM are
considered at the root.

For the hard-core model on a graph $G=(V,E)$, for a subset of
vertices $S$ and a fixed configuration $\rho$ on $S$, it is
equivalent to consider the subgraph $G'$ which we obtain for each
$v\in S$ that is fixed to be unoccupied we remove $v$ from $G$, and
for each $v\in S$ that is fixed to be occupied we remove $v$ and its
neighbors $N(v)$ from $G$. In this way we obtain the following
observation which will be useful for proving SSM holds.

\begin{observation}
\label{obs:SSM-WSM} For a graph $G=(V,E)$ and $v\in V$, SSM holds in
$G$ at vertex $v$ iff WSM holds for all subgraphs $G'$ (of $G$) at
vertex $v$. To be precise, by subgraphs we mean graphs obtained by
considering all subgraphs of $G$ and taking the component containing
$v$.
\end{observation}

\subsection{Self-Avoiding Walk Tree Representation}
\label{sec:saw}

Since our work builds on that of Weitz's,  we first describe the
self-avoiding walk (SAW) tree representation introduced in
\cite{Weitz}. Given $G=(V,E)$, we first fix an arbitrary ordering
$>_w$ on the neighbors of each vertex $w$ in $G$. For each  $v\in
V$, the tree $\Tsaw(G,v)$ is constructed as follows. Consider the
tree $T$ of self-avoiding walks originating from $v$, additionally
including the vertices closing a cycle as leaves of the tree.
 We then fix such leaves of $T$ to be
 occupied or unoccupied in the following manner.
 If a leaf vertex closes a cycle in $G$, say $w\to v_1 \to \dots v_\ell\to w$, then if $v_1>_{w} v_\ell$ we fix this leaf to be unoccupied, otherwise if $v_1<_{w}v_\ell$ we fix the leaf to be occupied.
Note, if the leaf is fixed to be unoccupied we simply
 remove that vertex from the tree.  If the leaf is fixed to be
 occupied, we remove that leaf and all of its neighbors, i.e.
 we remove the parent of that leaf from the tree.
The resulting tree is denoted as $\Tsaw=\Tsaw(G,v)$. See Figure
\ref{fig:saw} for an illustration of $\Tsaw$ for a particular
example.

\begin{figure}[h]
\begin{center}
\includegraphics[width=11cm,angle=0]{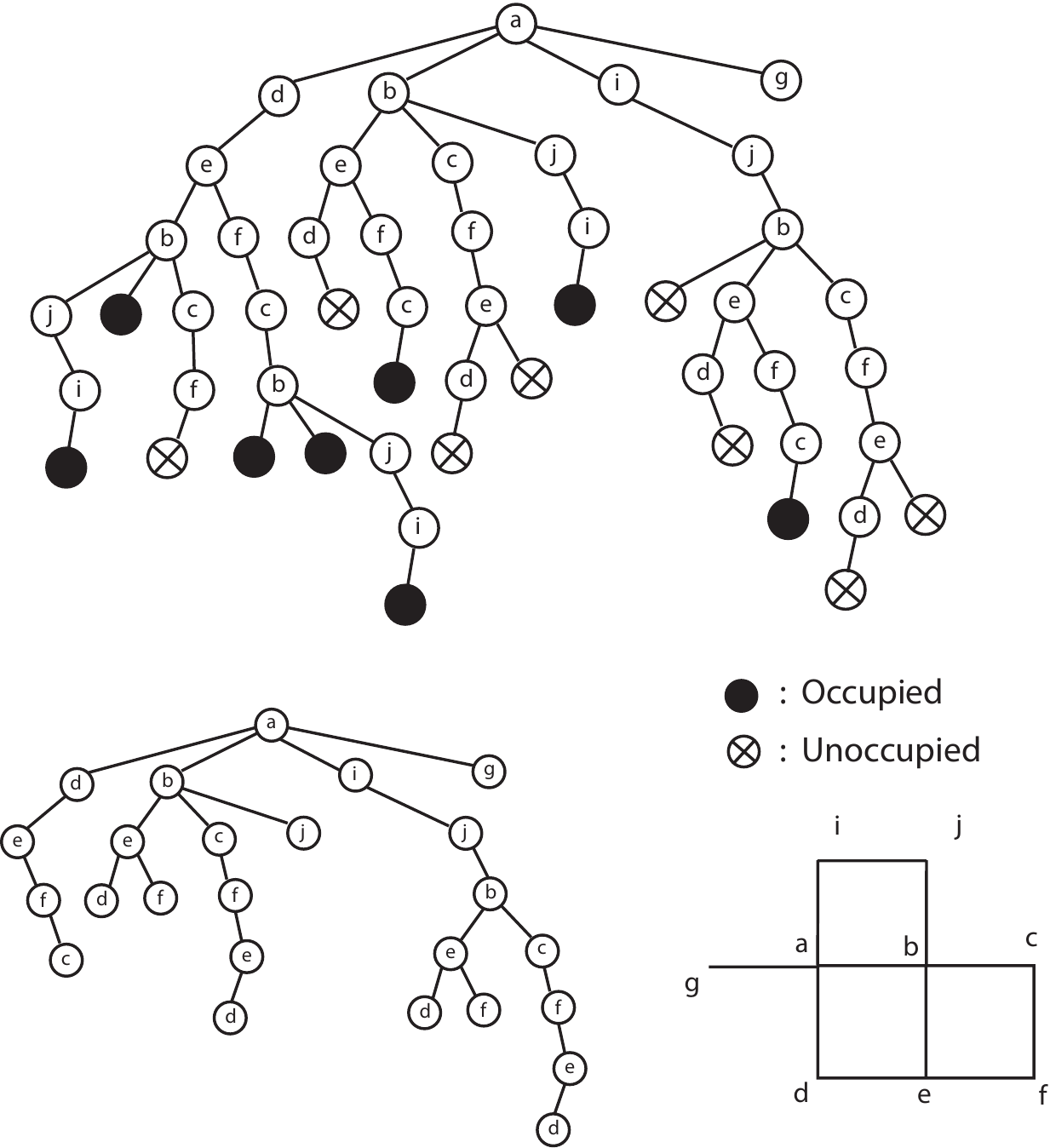}
\caption{Example of self-avoiding walk tree $\Tsaw$. The above tree
describes $\Tsaw(G,a)$ with occupied and unoccupied leaves, while
the below one is the same tree after removing those assigned leaves.
At each vertex, we consider the ordering $N>E>S>W$ of its neighbors
where $N,E,S,W$ represent the neighbors in the North, East, South,
West directions, respectively.} \label{fig:saw}
\end{center}
\end{figure}

Weitz \cite{Weitz} proves the following theorem for the hard-core
model, which shows that the marginal distribution at the root in
$\Tsaw(G,v)$ is identical to the marginal distribution for $v$ in
$G$. For a graph $G=(V,E)$, a subset $S\subset V$ and configuration
$\rho$ on $S$, for $\Tsaw=\Tsaw(G,v)$, let $\rho$ in $\Tsaw$ denote
the configuration on $S$ in $\Tsaw$ where for $w\in S$ every
occurrence of $w$ in $\Tsaw$ is assigned according to $\rho$.
\begin{theorem}[SAW Tree Representation, Theorem 3.1 in
\cite{Weitz}]\label{thm:saw} For any graph $G=(V,E)$, $v\in V$,
$\lambda>0$, and configuration $\bfrho$ on $S\subset V$, for
$T=\Tsaw(G,v)$ the following holds:
\begin{eqnarray*}
\alpha_{G,v}^{\bfrho} = \alpha_{T,v}^{\bfrho}.
\end{eqnarray*}
\end{theorem}
Note, the tree $\Tsaw(G,v)$ preserves the distance of vertices from
$v$ in $G$, which implies the following corollary.

\begin{corollary}\label{cor:weitz}
If  SSM holds with rate $\gamma$ for $\Tsaw(G,v)$ \mbox{ for all }
$v$,  then SSM holds for $G$ with rate~$\gamma$.
\end{corollary}

The reverse implication of Corollary \ref{cor:weitz} does not hold
since there are configurations on $S$ in $\Tsaw$ which are not
necessarily realizable in $G$. Observe that if $G$ has maximum
degree $\Delta$, any SAW tree of $G$ is a subtree of the regular
tree of degree $\Delta$.

\subsection{Our Proof Approach}

 In summary, Weitz \cite{Weitz} first shows (via Theorem \ref{thm:saw})
 that to prove SSM holds on a graph $G=(V,E)$,
 it suffices to prove SSM
 holds on the trees $\Tsaw(G,v)$, for all $v\in V$.
 Weitz then proves that the regular tree $\T_\Delta$ ``dominates''
 every tree of maximum degree $\Delta$
 in the sense that, for all trees of maximum degree $\Delta$,
 SSM holds when $\lambda<\lambda_c(\T_\Delta)$.
 We refine this second part of Weitz's approach.
  In particular, for graphs with extra structure, such as
  $G=\integers^2$, we bound $\Tsaw(\integers^2)$ by a tree $T^*$ that
is much closer to it than the regular tree $\T_\Delta$. We then
establish a criterion that achieves better bounds on SSM for trees
when the trees have extra structure.

The tree $T^*$ will be constructed in a regular manner so that we
can prove properties about it --  the construction of $T^*$ is
governed by a (progeny) $t\times t$ matrix $\bM$, whose rows
correspond to $t$ types of vertices, with the entry $M_{ij}$
specifying the number of children of type $j$ that a vertex of type
$i$ begets. We will then show a sufficient condition using entries
of $\bM$ which implies that SSM holds for $T^*$ and for any subgraph
of $T^*$, including $\Tsaw(\integers^2)$. The construction of $T^*$
is reminiscent of the strategy employed in \cite{A93,PT00} to upper
bound the connectivity constant of several lattice graphs, including
$\mathbb{Z}^2$. The derivation of our sufficient condition has some
inspiration from belief propagation algorithms.

As a byproduct of our proof that our new criterion implies SSM for
$T^*$, we get a new (and simpler) proof of the second part of
Weitz's approach, namely, that for all trees of maximum
degree~$\Delta$,
 SSM holds when $\lambda<\lambda_c(\T_\Delta)$.

\section{Branching Matrices and Strong Spatial Mixing}

As alluded to above,  we will utilize more structural properties of
self-avoiding walk trees. To this end, we consider families of trees
which can be recursively generated by certain rules; we then show
that such a general family is also analytically tractable.

\subsection{Definition of Branching Matrices}
\label{sec:branching-trees}

We say that the matrix $\bM$ is a $t\times t$ branching matrix if
every entry $M_{ij}$ is a non-negative integer. We say the maximum
degree of $\bM$ is $\Delta = \Delta(\bM) = \max_{1\le i\le t}
\sum_{1\le j\le t} M_{ij}$, the maximum row sum. Given a branching
matrix $\bM$, we define the following family of graphs. In essence,
it includes a graph $G$ if the self-avoiding walk trees of $G$ can
be generated by $\bM$.
\begin{definition}[Branching Family]
\label{def:branching-matrix} Given a $t\times t$ branching matrix
$\bM$, $\mathcal F_{\leq\bM}$ includes trees which can be generated
under the following restrictions:
\begin{itemize}
\item[$\circ$] Each vertex in tree $T\in \mathcal F_{\leq \bM}$ has its type $i\in \{1,\dots,
t\}$.
\item[$\circ$] Each vertex of type $i$ has at most $M_{ij}$ children of type $j$.
\end{itemize}
In addition, we use the notation $G=(V,E)\in\mathcal F_{\leq\bM}$
if $\Tsaw(G,v)\in \mathcal F_{\leq \bM}$ for all $v\in V$.
\vspace{0.1in}
\end{definition}

For example, the family $\mathcal F_{\leq \bM}$ with $\bM=[\Delta]$
includes the family of trees with maximum branching $\Delta$. On the
other hand, $\mathcal F_{\leq \bM}$ with $\bM=\begin{pmatrix}0 &
\Delta+1\\0 & \Delta\end{pmatrix}$ describes the family of graphs of
maximum degree $\Delta+1$, by assigning the root of tree
$T\in\mathcal F_{\leq \bM}$ to be of type 1 and the other vertices
of the tree to be of type 2. Note that if $\bM$ has maximum degree
$\Delta$, then every $G\in\FM$ also has maximum degree $\Delta$.

In this framework, Weitz's result establishing SSM for all graphs of
maximum degree $\Delta$ when $\lambda<\lambda_c(\T_\Delta)$ can be
stated as establishing SSM with uniform rate for all $G\in \mathcal
F_{\leq\bM}$ with $\bM=\begin{pmatrix}0 & \Delta\\0 & \Delta -
1\end{pmatrix}$; and we are interested in establishing its analogy
for general $\bM$. To this end, we will use the following notion of
SSM for $\bM$.

\begin{remark}
\label{rem:single-type} To establish SSM for $\bM$, it suffices to
prove that SSM holds with uniform rate for all trees in $\mathcal
F_{\leq \bM}$ due to Corollary \ref{cor:weitz}. In addition, note
that SSM holds for $\bM=\begin{pmatrix}0 & \Delta+1\\0 &
\Delta\end{pmatrix}$ if and only if it holds for $(\Delta)$ since
the root of a tree $T\in\mathcal F_{\leq \bM}$ is the only possible
vertex of type 1 in $T$.
\end{remark}

Finally, we define SSM for a branching matrix $\bM$.

\begin{definition}
Given a branching matrix $\bM$, we say SSM holds for $\bM$ if SSM
holds with uniform rate for all $G\in \mathcal F_{\leq\bM}$.
\end{definition}

\begin{remark}
\label{rem:SSM} To establish SSM for $\bM$, it suffices to prove
that SSM holds with uniform rate for all trees in $\mathcal F_{\leq
\bM}$ due to Corollary \ref{cor:weitz}.
\end{remark}

\subsection{Implications of SSM}
\label{sec:implications}

We present a new approach for proving SSM for a branching matrix
$\bM$.  There are multiple consequences of SSM for  $\bM$ as
summarized in the following theorem. We first state some definitions
needed for stating the theorem.

Following Goldberg et al. \cite{GMP} we use the following variant of
amenability for infinite graphs. Here we consider an infinite graph
$G=(V,E)$. For $v\in V$ and a non-negative integer $d$, let
$\Ball_d(v)$ denote the set of vertices within distance $\le d$ from
$v$, where distance is the length of the shortest path. For a set of
vertices $S$, the (outer) boundary and neighborhood amenability are
defined, respectively,  as:
\[ \partial S := \{ w\in V: w\notin S, \mbox{ and $w$ has a neighbor $y\in S$}\} \ \ \mbox{ and }
 \ \  r_d = \sup_{v\in V} \frac{|\partial\Ball_d(v)|}{|\Ball_d(v)|}\,.
\]
The infinite graph is said to be {\em neighborhood-amenable} if
$\inf_d r_d = 0$.

Now we can state the following theorem detailing the implications of
SSM of interest to us.

\begin{theorem}
\label{thm:implications} For a $t\times t$ branching matrix  $\bM$,
if SSM holds for $\bM$ then the following hold:
\begin{enumerate}
\item
\label{imp:SSM} For every $G\in\FM$, SSM holds on $G$.
\item
\label{imp:unique} For every infinite graph $G\in\FM$, there is a
unique infinite-volume Gibbs measure on $G$.
\item
\label{imp:FPAS} If $\bM$ has maximum degree $\Delta$, if $t=O(1)$
and $\Delta=O(1)$, then for every (finite) $G\in\FM$, Weitz's
algorithm \cite{Weitz} gives an FPAS for approximating the partition
function $Z(G)$.
\item
\label{imp:Glauber} For every infinite $H\in\FM$ which is
neighborhood-amenable, for every finite subgraph $G=(V,E)$ of $H$,
the Glauber dynamics has $O(n^2)$ mixing time. Moreover, if
$H=\integers^d$ for constant $d$, then for every finite subgraph
$G=(V,E)$ of $H$, the Glauber dynamics has $O(n\log{n})$ mixing
time.
\end{enumerate}
\end{theorem}

\begin{proof}
Part \ref{imp:SSM} is by the definition of SSM for $\bM$. The
uniqueness result follows from the fact that the infinite-volume
extremal Gibbs measures on the infinite graph $G$ can be obtained by taking
limits of finite measures, see Georgii \cite{Georgii} for an
introduction to infinite-volume Gibbs measures, and see Martinelli
\cite{Marlec} for Part \ref{imp:unique}. Part \ref{imp:FPAS}
immediately follows from the work of Weitz \cite{Weitz}. Finally,
for Part \ref{imp:Glauber}, there is a long line of work showing
that for the integer lattice $\integers^d$ in fixed dimensions, for
the Ising model SSM on $\integers^d$ implies $O(n\log{n})$ mixing
time of the Glauber dynamics on finite subregions of $\integers^d$,
e.g., see Cesi \cite{Cesi} and Martinelli \cite{Marlec} (and the
references therein) for recent results on this problem.  These
results for the Ising model are typically stated for a general class
of models, but that class does not include models with hard
constraints, such as the hard-core model studied here.  Dyer et al.
\cite{DSVW} showed a simpler proof for the hard-core model that
utilizes the monotonicity of the model. We use this result of
\cite{DSVW} in Theorem \ref{thm:z2} to get $O(n\log{n})$ mixing time
for subregions of $\integers^2$.  Goldberg et al. \cite[Theorem
8]{GMP} showed that for $k$-colorings, if SSM holds for an infinite graph $G$
that is neighborhood-amenable, the Glauber dynamics has $O(n^2)$
mixing time for all finite subgraphs of $G$.  Their proof holds for
the hard-core model which implies Part~\ref{imp:Glauber}.
\end{proof}

\section{Establishing SSM for Branching Matrices}
\label{sec:reg}

In this section we present a sufficient condition implying SSM for
the family of trees generated by a branching matrix. As a
consequence of the approach presented in this section we get a
simpler proof of Weitz's result \cite{Weitz} implying SSM for all
graphs with maximum degree $\Delta$ when
$\lambda<\lambda_c(\T_\Delta)$. We then apply the condition
presented in this section to $\mathbb{Z}^2$ in Section
\ref{sec:grid}.

To show the decay of influence of a boundary condition $\bfrho$, a
common strategy is to prove some form of contraction for the
`one-step' iteration given in \eqref{eq:one-step1} below. More generally,
we will prove such a contraction for an appropriate set of
`statistics' of the unoccupied marginal probability.

A {\em statistic} of the univariate parameter $x\in\lbrack a,b]$ is
a monotone (i.e., strictly increasing or decreasing) function
$\varphi:[a,b]\rightarrow\mathbb{R}$. For a $t\times t$ branching
matrix $\bM$ we consider a set of $t$ statistics
$\varphi_1,\dots,\varphi_t$, one for each type.
For the simpler case when $\bM=[\Delta]$ and
hence $t=1$, we have a single statistic $\varphi$. Our aim is proving
contraction for an appropriate set of statistics of the probability
that the root of a tree is unoccupied.

We first focus on the case of a single type.
Consider a tree $T=(V,E)\in \mathcal F_{\leq\bM}$ with root $r$.
For $v\in V$, let
$N(v)$ denote the children of $v$,
and let $d(v):=|N(v)|$ the number of children.
Let $T_v$ denote the subtree rooted at $v$. We will
analyze the unoccupied probability for a vertex $v$, but $v$ will
always be the root of its subtree.  Hence, to simplify the notation, for
a boundary condition $\bfrho$ on $S\subset V$, let
$\alpha_{v}^{\bfrho}= \alpha_{T_v,v}^{\bfrho}$.

A straightforward recursive calculation with the partition function
leads to the following relation:
\begin{equation}
\alpha^{\bfrho}_{v}=\left\{
\begin{array}
[c]{cc} \frac{1}{1+\lambda} & \mbox{if }N(v)=\emptyset
\vspace{0.05in}\\
\frac1{1+\lambda\prod_{w\in N(v)}\alpha_{w}^{\bfrho}  } &
\mbox{otherwise.$~~$}
\end{array}\right.
\label{eq:one-step1}
\end{equation}
Note, the unoccupied probability always lies in the
interval $I:=\left[\frac1{1+\lambda},1\right]$, i.e., for all $v$, all $\rho$,
$\alpha_v^\rho \in I$.

For $v\in V$, let $m_v^{\bfrho} :=
\varphi(\alpha^{\bfrho}_v)$ be the `message' at vertex $v$.
The messages satisfy the following recurrence:
$$m_v^{\bfrho} =
\varphi\left(\frac1{1+\lambda\prod_{w\in
N(v)}\alpha_w^{\bfrho}}\right)
=\varphi\left(\frac1{1+\lambda\prod_{w\in
N(v)}\varphi^{-1}(m_w^{\bfrho})}\right).$$

Our aim is to prove uniform contraction of the messages on
all trees $T\in\FM$.
To this end, we will consider a more general set of messages.
Namely, we consider
messages $m_1,\dots,m_\Delta$ where
for every $1\le i\le\Delta$,
$m_i=\phi(\alpha_i)$ and $\alpha_i\in I:=\left[\frac1{1+\lambda},1\right]$.
This set of tuples
$\alpha_1,\dots,\alpha_\Delta\in I$
contains all of the tuples obtainable on a tree.

For $\alpha_1,\dots,\alpha_\Delta\in I$, let
$m_i=\phi(\alpha_i), 1\le i\le\Delta$, and
let $$F(m_1,\ldots,m_{\Delta}):=\varphi\left(
\frac1{1+\lambda\prod_{i=1}^{\Delta}\varphi^{-1}(m_i)}\right).$$

Ideally, we would like to establish the following contraction:
there exists a $0<\gamma<1$ such that for all
$\alpha_1,\dots,\alpha_\Delta,\alpha'_1,\dots,\alpha'_\Delta\in I$,
\[ |F(m_1,\dots,m_\Delta) - F(m'_1,\dots,m'_\Delta)|
\leq\gamma \max_{1\le i\le\Delta} |m_i - m'_i|,
\]
where
$m_i=\phi(\alpha_i)$ and $m'_i=\phi(\alpha'_i)$.
We will instead show that
the following weaker condition suffices.
Namely, that
the desired contraction holds for all $|\alpha_i-\alpha'_i|\leq\eps$ for
some $\eps>0$.
This is equivalent to the following condition.

\begin{definition}
Let $I=\left[\frac1{1+\lambda},1\right]$. For the branching matrix
$\bM=[\Delta]$, we say that {\em Condition~($\simplerstarcond$)}  is
satisfied if for all $\alpha_1,\ldots,\alpha_{\Delta}\in
I$, by setting  $m_i=\varphi(\alpha_i)$ for  $1\le i\le\Delta$, the
following holds:
\[\tag{$\star$}
 \left\Vert \nabla F\left(
m_1,\ldots,m_{\Delta}\right) \right\Vert _{1}
=
\sum_{i=1}^{\Delta} \left\vert \frac{\partial
F\left(m_1,\ldots,m_{\Delta}\right)}{\partial m_{i}} \right\vert
<1.
\]
\end{definition}

Let us now consider a natural generalization of the above notion
for a branching matrix with multiple types.
Let $\bM$ be a $t\times t$ branching matrix.  For $1\le \ell\le t$,
let $\Delta_\ell=\sum_{k=1}^t M_{\ell k}$ denote the maximum number
of children of a vertex of type $\ell$.
Once again, consider a tree $T=(V,E)\in \mathcal F_{\leq\bM}$ with root $r$.
For $v\in V$, let $t(v)$ denote its type.
As before, $N(v)$ are the children of $v$, $d(v)$ is the number of children of
$v$, and for a boundary condition $\bfrho$ on $S\subset V$,
 $\alpha_{v}^{\bfrho}$ is the unoccupied probability for $v$ in
 the tree $T_v$ under $\rho$.

 The recursive calculation in \eqref{eq:one-step1} for
 $\alpha_v$ in terms of $\alpha_w, w\in N(v)$, still holds.
For the case of multiple types, for $v\in V$, let $m_v^{\bfrho} :=
\varphi_{t(v)}(\alpha^{\bfrho}_v)$ be the message at vertex $v$.
The messages satisfy the following recurrence:
$$m_v^{\bfrho} =\varphi_{t(v)}\left(\frac1{1+\lambda\prod_{w\in
N(v)}\varphi^{-1}_{t(w)}(m_w^{\bfrho})}\right).$$

For each type $1\le \ell\le t$, we consider contraction of
messages derived from all $\alpha_1,\dots,\alpha_{\Delta_\ell}\in I$.
We need to identify the type of each these quantities $\alpha_i$
in order to determine the appropriate statistic to apply.
The assignment of types needs to be consistent with the
branching matrix $\bM$.  Hence, let
$s_\ell:\{1,\dots,\Delta_\ell\}\rightarrow\{1,\dots,t\}$ be the following assignment.
Let $M_{\ell,\le 0} = 0$ and for $1\le i\le t$, let $M_{\ell,\le i} = \sum_{k=1}^{i} M_{\ell,k}$.
For $1\le i\le t$, for $M_{\ell,\le i-1}< j\le M_{\ell,\le i}$, let $s_\ell(j) = i$.

For type $1\le \ell\le t$, for $\alpha_1,\dots,\alpha_{\Delta_\ell}\in I$,
set $m_j = \phi_{s_\ell(j)}(\alpha_j), 1\le j\le \Delta_\ell$, and
let
$$F_{\ell}(m_1,\ldots,m_{\Delta_\ell}):=\varphi_\ell\left(
\frac1{1+\lambda\prod_{j=1}^{\Delta_\ell}\varphi^{-1}_{s_\ell(j)}(m_j)}\right).$$

Note,
\begin{equation}
\label{eq:rec-F}
m^{\bfrho}_v=F_{t(v)}\left(m^{\bfrho}_{w_1},\ldots,m^{\bfrho}_{w_{d(v)}}\right)
\quad\mbox{where}\quad N(v) = \{w_1,\dots,w_{d(v)}\}.\footnote{Strictly speaking, $F_\ell$ requires $\Delta_\ell$ arguments, so
for \eqref{eq:rec-F} to hold in the case when $d(v)<\Delta_\ell$ we
can simply add additional arguments corresponding to $\alpha=1$,
which fixes these additional vertices to be unoccupied (and therefore absent).}
\end{equation}

We generalize Condition~($\simplerstarcond$)
to branching matrices with
multiple types by allowing a weighting of the types by parameters
$c_1,\dots,c_t$.

\begin{definition}
Let $I=\left[\frac1{1+\lambda},1\right]$. For a $t\times t$
branching matrix $\bM$,  we say that  {\em Condition~($\starcond$)}
is satisfied if there exist $c_1,\dots,c_t$,
such that for all $1\le \ell\le t$, for all
$\alpha_1,\ldots,\alpha_{\Delta_\ell}\in I$,
by setting $m_i=\varphi_{s_\ell(i)}(\alpha_i)$ for  $1\le i\le \Delta_\ell$,
the following holds:
\[\tag{$\starcond$}
\sum_{i=1}^{\Delta_\ell}c_{s_\ell(i) }\left\vert \frac{\partial
F_{\ell}\left(m_1,\ldots,m_{\Delta_\ell}\right)}{\partial m_{i}} \right\vert
<c_{\ell}.
\]
\end{definition}

The following lemma establishes a sufficient condition so that SSM
holds for~$\bM$.

\begin{lemma}
\label{pro:recM} For a $t\times t$ branching matrix $\bM$, if for
every $1\le \ell\le t$, $\varphi_{\ell}$ is continuously
differentiable on the interval $I=\left[\frac1{1+\lambda},1\right]$
and $\inf\limits_{x\in I} |\varphi_\ell^{\prime}(x)|>0$, and if
Condition ($\simplerstarcond$) is satisfied for $t=1$ or Condition
($\starcond$) is satisfied for $t\geq 2$ then SSM holds for $\bM$,
and hence the conclusions of Theorem \ref{thm:implications} follow.
\end{lemma}

\begin{proof}
For a tree $T=(V,E)$ with root $r$, let $\alpha^{+}_{L,r}$ and
$\alpha^{-}_{L,r}$ denote the marginal probabilities that the root
of $T$ is unoccupied conditional on the vertices at level $L$ (i.e.,
distance $L$ from the root) being occupied and unoccupied,
respectively.

The main result for proving Lemma \ref{pro:recM} is that there exist
$\gamma<1$ and $L_0<\infty$ such that for every tree $T\in \mathcal
F_{\leq \bM}$
 and every integer $L\geq L_0$,
\begin{equation}\label{eq:wsm-general}
\left|\alpha^{+}_{L,r}-\alpha^{-}_{L,r}\right|
~\leq~\gamma^L.\end{equation}

We first explain why \eqref{eq:wsm-general} implies Lemma
\ref{pro:recM} and then we prove \eqref{eq:wsm-general}. Consider a
tree $T=(V,E)$ with root $r$, and a boundary condition $\bfrho$ on
$S\subset V$. Set $L=\dist(r,S)$ as the distance of $S$ to the root
of $T$. The hard-core model on bipartite graphs has a monotonicity
of boundary conditions (c.f., \cite{DSVW}) which implies that for
odd $L$, $\alpha^{+}_{L,r}\geq \alpha_r^{\bfrho} \geq
\alpha^{-}_{L.r}$, and for even $L$, $\alpha^{+}_{L,r}\geq
\alpha_r^{\bfrho} \geq \alpha^{-}_{L,r}$. Hence, for any pair of
boundary conditions $\bfrho$ and $\bfeta$ on $S$,
\[
\left|\alpha_r^{\bfrho} - \alpha_r^{\bfeta}  \right| \leq
\left|\alpha^{+}_{L,r}-\alpha^{-}_{L,r}\right|.
\]
Therefore, by the definition of WSM in Definition \ref{def:wsm},
proving \eqref{eq:wsm-general} implies WSM for $T$. Since this holds
for all $T'\in  \mathcal F_{\leq \bM}$, by Observation
\ref{obs:SSM-WSM}, it implies SSM for all $T'\in \mathcal F_{\leq
\bM}$, which, by Remark~\ref{rem:SSM}, implies SSM for $\bM$.

We now turn our attention to proving \eqref{eq:wsm-general}. Fix a
$t\times t$ branching matrix $\bM$ and consider a tree $T=(V,E)\in
\mathcal F_{\leq \bM}$ with root $r$. Given $y\in [0,1]$, let
$\beta_{L,v}(y)$ denote the marginal probability that the root of
$T_v$ is unoccupied given all of the vertices at level $L$ (in
$T_v$) are assigned marginal probability $y$ of being unoccupied
(conditional on its parent being unoccupied). Intuitively,
$\beta_{L,v}(y)$ can be thought as the marginal probability
conditioned on a `fractional' boundary configuration at level $L$.
As in \eqref{eq:one-step1}, $\beta_{L,r}(y)$ satisfies the following
recurrence for $y\in [0,1]$:
\begin{equation}\label{eq:defalphau-1}
\beta_{L,r}\left(y\right)  =\left\{
\begin{array}{cc}
y & \text{if } L=0,$\qquad\qquad\qquad~$ \vspace{0.05in}\\
\frac{1}{1+\lambda} & \mbox{if } L>0 \mbox{ and } N(r)=\emptyset,
\vspace{0.05in}\\
\frac{1}{1+\lambda\prod_{w\in N(r)}\beta_{L-1,w}\left( y\right)} &
\text{otherwise}.$\qquad\qquad\quad~$
\end{array}\right.
\end{equation}

From \eqref{eq:defalphau-1} and \eqref{eq:one-step1},
it follows that $\alpha^{+}_{L,r}= \beta_{L,r}\left(  1\right)$ and
$\alpha^{-}_{L,r}= \beta_{L,r}\left(  0\right)$. Hence, in order to
analyze the messages for $\alpha^{+}_{L,r}$ and $\alpha^{-}_{L,r}$,
we will analyze the messages for $\beta_{L,r}(y)$. Therefore, for
$v\in V$, let $m_{L,v}\left(y\right)=\varphi_{t(v)}\left(
\beta_{L,v}\left( y\right)  \right)$. Analogous to \eqref{eq:rec-F},
we now have that:
\[
m_{L,r}(y)=F_{t(r)}\left(
m_{L-1,w_1}\left(y\right),\ldots,m_{L-1,w_{d(r)}}\left(y\right)\right)
\qquad\mbox{where}\quad N(r) = \{w_1,\dots,w_{d(r)}\}.
\]
Observe that for all $y\in[0,1]$, all $L>0$, all $v\in V$,
$\beta_{L,v}\left( y\right) \in I =\left[ \frac1{
1+\lambda},1\right]$, and hence we can use Condition~($\starcond$)
to analyze $m_{L,r}$.

 Using the fact that $\beta_{L,v}\left( y\right)  $ and
$m_{L,v}\left(  y\right)  $ are continuously differentiable for
$y\in\left[  0,1\right]$, we have that for $L>0$,
\begin{eqnarray*}
\left\vert \alpha^{+}_{L,r}-\alpha^{-}_{L,r}
\right\vert&=&\left\vert \beta_{L,r}(1)-\beta_{L,r}(0) \right\vert
~\leq~\int_{0}^{1}\left\vert \frac{\partial \beta_{L,r}\left(
y\right)}{\partial y}
  \right\vert dy ~\leq~\frac{\int_{0}^{1}\left\vert
\frac{\partial m_{L,r}\left(  y\right)}{\partial y} \right\vert
dy}{\inf\limits_{x\in I }\left\vert \varphi^{\prime}_{t(r)}\left(
x\right)  \right\vert }.
\end{eqnarray*}

By the hypothesis of Lemma \ref{pro:recM}, we know that $
\left\vert\varphi^{\prime}_{t(r)}\left(  x\right)\right\vert >0$.
Therefore, to prove the desired conclusion \eqref{eq:wsm-general},
it suffices to prove that there exist constants $K<\infty$ and
$\eta<1$ such that for every tree $T\in  \mathcal F_{\leq \bM}$ with
root $r$, all $L>0$,
\begin{equation}
\left\vert \frac{\partial m_{L,r}\left(  y\right)}{\partial y}
\right\vert \leq c_{t(r)}K \eta^{L-1}. \label{eq:induc-1}
\end{equation}
Note that $K$ and $\eta$ should be independent of $T$ and $L$, but
may depend on $\lambda, \varphi_1,\ldots,\varphi_t$ and $c_1,\ldots,c_t$.
The constant $K$ will be the following:
\[
K~:=~\frac{\lambda\Delta\max\limits_{1\le \ell\le
t}\sup\limits_{x\in I }\left\vert \varphi^{\prime}_\ell\left(
x\right) \right\vert }{\min\limits_{1\le\ell\le t} c_\ell},
\]
and the constant $\eta$ will be the constant implicit in Condition
($\starcond$).

We will show \eqref{eq:induc-1} by induction on $L$. First we verify
the base case $L=1$. In this case,
\[ m_{L,r}(y)=
\varphi_{t(r)}\left( \beta_{L,r}\left( y\right)  \right) =
\varphi_{t(r)}\left(\frac1{1+\lambda y^{d(r)}}\right).
\]
Thus,
\begin{align*}
\left\vert \frac{\partial m_{L,r}\left(  y\right)}{\partial y}
\right\vert 
& = \left\vert \frac{\partial \varphi_{t(r)}\left(\frac1{1+\lambda
y^{d(r)}}\right)}{\partial y}
  \right\vert &
\mbox{ since $L=1$}
\\&\leq
\sup\limits_{x\in I }\left\vert \varphi^{\prime}_{t(r)}\left(
x\right) \right\vert \sup\limits_{y\in \left[ 0,1\right]
}\frac{\lambda d(r)y^{d(r)-1}}{\left( 1+\lambda y^{d(r)}\right)
^{2}}
 & \mbox{ by the chain rule}
\\&\leq
\sup\limits_{x\in I }\left\vert \varphi^{\prime}_{t(r)}\left(
x\right) \right\vert \lambda d(r)
\\
&\leq \sup\limits_{x\in I }\left\vert \varphi^{\prime}_{t(r)}\left(
x\right) \right\vert \lambda\Delta
\\
&\leq c_{t(r)} K
 & \mbox{ by the definition of $K$.}
\end{align*}
This completes the analysis of the base case.

Now we proceed toward establishing the necessary induction step
using the inductive hypothesis.
We have that
\begin{align}
\left\vert \frac{\partial m_{L,r}\left(y\right)}{\partial
y}\right\vert &=\left\vert \frac{\partial F_{t(r)}\left(
m_{L-1,w_1}\left(y\right),\ldots,m_{L-1,w_{d(r)}}\left(y\right)\right)}{\partial
y}\right\vert\nonumber
\\&=\left\vert \sum_{i=1}^{d(r)} \frac{\partial F_{t(r)}\left(
m_{1},\ldots,m_{d(r)}\right)}{\partial m_i}\cdot \frac{\partial
m_{L-1,w_i}(y)}{\partial y}\right\vert&\mbox{where
}m_i:=m_{L-1,w_i}(y)\nonumber
\\
&=\left\vert \sum_{i=1}^{d(r)} c_{t(w_i)}\frac{\partial F_{t(r)}\left(
m_{1},\ldots,m_{d(r)}\right)}{\partial m_i}\cdot
\frac1{c_{t(w_i)}}\frac{\partial m_{L-1,w_i}(y)}{\partial
y}\right\vert\nonumber 
  \end{align}
\begin{align}
&=\left\vert \sum_{i=1}^{d(r)} c_{t(w_i)}\frac{\partial
F_{t(r)}\left( m_{1},\ldots,m_{d(r)}\right)}{\partial
m_i}\right\vert\nonumber\\
&\qquad\qquad\qquad\qquad\qquad\times \max_{1\leq i\leq
d(r)}\frac1{c_{t(w_i)}}\left\vert\frac{\partial
m_{L-1,w_i}(y)}{\partial y}\right\vert& \mbox{by H\"{o}lder's
inequality.} \label{ind-hyp}
  \end{align}
From ($\starcond$), there exists a universal constant $\eta<1$ such
that \[\left\vert \sum_{i=1}^{d(r)} c_{t(w_i)}\frac{\partial
F_{t(r)}\left( m_{1},\ldots,m_{d(r)}\right)}{\partial
m_i}\right\vert<\eta\, c_{t(r)}.\] Therefore, it follows that
\begin{align*}
\left\vert \frac{\partial m_{L,r}\left(y\right)}{\partial
y}\right\vert &\leq
 \eta\, c_{t(r)} \cdot \max_{1\leq i\leq
d(r)}\frac1{c_{t(w_i)}}\left\vert\frac{\partial
m_{L-1,w_i}(y)}{\partial y}\right\vert
& \mbox{by \eqref{ind-hyp} and the definition of $\eta$}\\
&\leq c_{t(r)} K\eta^{L-1}  & \mbox{by the inductive hypothesis.}
\end{align*}
 This completes the proof of \eqref{eq:induc-1}, and hence that of Lemma
\ref{pro:recM}.
\end{proof}

\subsection{Reproving Weitz's Result of SSM for Trees}

In this section, we aim at finding a good choice of statistics.
First we find such a statistic for the case $\bM=[\Delta]$, i.e.,
the case of a single type, which enables us to reprove Weitz's
result \cite{Weitz} that when $\lambda<\lambda_c(\T_\Delta)$ SSM
holds for every tree of maximum degree $\Delta$.

Using Lemma \ref{pro:recM} (and the simpler condition
($\simplerstarcond$) for the case of a single type) we obtain a
simpler proof of Weitz's result \cite{Weitz} that for every tree $T$
with maximum degree $\Delta+1$ (hence, for every graph $G$ of
maximum degree $\Delta+1$) and for all
$\lambda<\lambda_c(\T_{\Delta+1})=\Delta^{\Delta}/(\Delta-1)^{\Delta+1}$,
SSM holds on $T$ (and on~$G$).

\begin{theorem}
\label{thm:reprove-Weitz} Let $\varphi(x)=\frac{1}{s}\log\left(
\frac{x}{s-x}\right) $ where $s=\frac{\Delta+1}{\Delta}$. Then,
Condition ($\simplerstarcond$) holds for $\bM=[\Delta]$ and
$\lambda<\lambda_c(\T_{\Delta+1})$. Consequently, SSM and the
conclusions of Theorem \ref{thm:implications} hold for
$\bM=\begin{pmatrix}0 & \Delta+1\\0 & \Delta\end{pmatrix}$ and
$\lambda<\lambda_c(\T_{\Delta+1})$.
\end{theorem}

\begin{proof}
First, a straightforward calculation implies that
$$\left\vert\frac{\partial F}{\partial m_i}\right\vert
=\frac{1-\alpha}{s-\alpha}(s-\alpha_i),$$ where
$\alpha_i=\varphi^{-1}(m_i)$ and
$\alpha=\left(1+\lambda\prod_{i=1}^{\Delta}\alpha_i\right)^{-1}$.

Hence, we have
\begin{align}
\nonumber \left\Vert \nabla F \right\Vert _{1} &=
\sum_{i=1}^{\Delta} \left\vert\frac{\partial F}{\partial
m_i}\right\vert
\\
\nonumber & =
\sum_{i=1}^{\Delta} \frac{1-\alpha}{s-\alpha}(s-\alpha_i)
\nonumber 
\end{align}
\begin{align}
&\leq \frac{1-\alpha}{s-\alpha}\ \Delta\
\left(s-\left(\prod_{i=1}^{\Delta}\alpha_i\right)^{1/\Delta}\right)
& \mbox{by the arithmetic-geometric mean inequality} \\
\label{ineq:gradient} &=\frac{1-\alpha}{s-\alpha}\ \Delta\
\left(s-\left(\frac{1-\alpha}{\lambda\alpha}\right)^{1/\Delta}\right).
\end{align}

We now use the following technical lemma.

\begin{lemma}\label{lem:tec}
\label{lem:fineq}$$\max_{x\in[0,1]} \frac{( 1-x) \left(
1+\frac1\Delta-( \frac{1-x}{\lambda x}) ^{\frac1\Delta}\right)
}{1+\frac1\Delta-x} ~\leq~\frac{\omega }{1+\omega},$$where $\Delta$
is a positive integer and $\omega$ is the unique solution to
$\omega( 1+\omega) ^{\Delta}=\lambda$.
\end{lemma}

Using the above inequality \eqref{ineq:gradient} with Lemma
\ref{lem:tec}, we have that:
\[  \left\Vert \nabla F \right\Vert _{1}<1
\ \ \mbox{  if }  \ \ \frac{\omega}{1+\omega}\cdot \Delta~<~1,
\] where $\omega$ is
the unique solution of $\omega(1+\omega)^{\Delta}=\lambda$. This
leads to the desired condition
$\lambda<\lambda_c(\T_{\Delta+1})=\Delta^{\Delta}/(\Delta-1)^{\Delta+1}$
so that SSM holds for $\bM = [\Delta]$. As we noted in Remark
\ref{rem:single-type}, this is equivalent to SSM for
$\bM=\begin{pmatrix}0 & \Delta+1\\0 & \Delta\end{pmatrix}$.
This completes the proof of Theorem \ref{thm:reprove-Weitz}.
\end{proof}

\begin{proof}[Proof of Lemma \ref{lem:tec}]
Let $\Phi_{\Delta}(  x)  =(  \frac{1-x}{\lambda x}) ^{\frac1\Delta}$
and $f( x) =\frac{(  1-x)  ( 1+\frac1\Delta-\Phi_{\Delta}(  x)  )
}{1+\frac1\Delta-x}$. Since $\Phi_{\Delta}^{\prime}(  x)
=-\frac{\Phi_{\Delta}( x) }{\Delta x( 1-x)  }$, $\Phi_{\Delta}$ is a
decreasing function in $[0,1]$ such that $\Phi_{\Delta}( 0)
=+\infty$ and $\Phi_{\Delta}( 1)  =0$. Therefore it has a unique
fixed point that can be shown to be $\bar{x}=\frac{1}{1+\omega}$.
Moreover, it is the case that $\Phi_{\Delta}(x)  >x$ if and only if
$x<\bar{x}$. To prove Lemma \ref{lem:tec}, we notice that
$f^{\prime}( x) =\frac{(  1+\frac1\Delta)  (  \Phi(  x) -x) }{\Delta
x( 1+\frac1\Delta-x) ^{2}}$, hence $f^{\prime}( x)  >0$ for
$x<\bar{x}$ and $f^{\prime}( x) <0$ for $x>\bar{x}$. This implies
that $f$ has a maximum at $\bar{x}$, namely $f(  \bar {x})
=\frac{\omega}{1+\omega}$.
\end{proof}

\subsection{DMS Condition: A Sufficient Criterion}

Theorem \ref{thm:reprove-Weitz} suggests choosing
$\varphi_{j}(x)=\frac{1}{s_{j}}\log\left( \frac{x}{s_{j}-x}\right) $
with appropriate parameters $s_j$ for a general branching matrix
$\bM$. Under this choice, we obtain the following condition for SSM.

\begin{definition}
[DMS Condition] Given a $t\times t$ branching matrix $\bM$ and
$\lambda^*>0$, for $s_{1},\ldots,s_{t}
>1$ and $\bc=\left( c_{1},\ldots,c_{t}\right)
>0$, let $\bD$ and $\bS$ be the diagonal matrices defined as
\[
D_{jj}=\sup_{\alpha\in\left[  \frac1{1+\lambda^*},1\right]
}\frac{\left( 1-\alpha\right) \left(  1-\theta_{j}\left(
\frac{1-\alpha}{\lambda^*\alpha}\right)^{1/\Delta_{j} }\right)
}{s_{j}-\alpha}\text{\quad\quad and \quad\quad}S_{jj}=s_{j}\text{,}
\]
where
\[
\theta_{j}:=\frac{\left(  \prod_{\ell}c_{\ell}^{M_{j\ell}}\right)
^{1/\Delta_{j}}}{\sum
_{\ell}c_{\ell}s_{\ell}M_{j\ell}/\Delta_{j}}\text{ \quad\quad
and\quad\quad}\Delta_{j}=\sum _{\ell}M_{j\ell}\text{.}
\]
We say the {\em DMS Condition holds for $\bM$ and $\lambda^*$} if
there exist $s_{1},\ldots,s_{t}>1$ and $\bc>0$ such that:
\[
\left(  \bD\bM\bS\right)  \bc<\bc.
\]

\end{definition}
\begin{theorem}\label{cor:dms}
If the DMS Condition holds for $\bM$ and $\lambda^*>0$, then
Condition ($\starcond$) holds with the choice of
$\varphi_{j}(x)=\frac{1}{s_{j}}\log\left( \frac{x}{s_{j}-x}\right) $
for all $\lambda\leq \lambda^*$. Consequently, SSM and the
conclusions of Theorem~\ref{thm:implications} hold for $\bM$ and all
$\lambda\leq\lambda^*$.
\end{theorem}

\begin{proof}
First, one can check that
$$\left\vert\frac{\partial F_j}{\partial m_i}\right\vert
=\frac{1-\alpha}{s_j-\alpha}(s_{j_i}-\alpha_i),$$ where
$\alpha_i=\varphi^{-1}_{j_i}(m_i)$ and
$\alpha=\frac1{1+\lambda\prod_{i=1}^{\Delta_j}\alpha_{i}}$.

Hence, it follows that
\begin{align*}
\sum_{i=1}^{\Delta_{j}}c_{j_i }\left\vert \frac{\partial
F_{j}}{\partial m_{j}} \right\vert &=
\frac{1-\alpha}{s_j-\alpha}\sum_{i=1}^{\Delta_{j}}c_{j_i}(s_{j_i}-\alpha_i)
\\
&\leq \frac{1-\alpha}{s_j-\alpha}\left(
\sum_{i=1}^{\Delta_j}c_{j_i}s_{j_i}-\Delta_j\left(\prod_{i=1}^{\Delta_j}c_{j_i}\alpha_i\right)^{1/\Delta_j}\right)
\ \ \ \mbox{by the arithmetic-geometric mean ineq.} \\
&= \frac{1-\alpha}{s_j-\alpha}\left(
\sum_{i=1}^{\Delta_j}c_{j_i}s_{j_i}-\Delta_j\left(\prod_{i=1}^{\Delta_j}c_{j_i}\right)^{1/\Delta_j}
\left(\frac{1-\alpha}{\lambda\alpha}\right)^{1/\Delta_j}\right)\\
&=\frac{1-\alpha}{s_j-\alpha}\left(
1-\theta_j\left(\frac{1-\alpha}{\lambda\alpha}\right)^{1/\Delta_j}\right)
\sum_{i=1}^{\Delta_j}c_{j_i}s_{j_i} \ \ \ \mbox{by the definition of
$\theta_j$}
\\
&\leq\frac{1-\alpha}{s_j-\alpha}\left(
1-\theta_j\left(\frac{1-\alpha}{\lambda^*\alpha}\right)^{1/\Delta_j}\right)
\sum_{i=1}^{\Delta_j}c_{j_i}s_{j_i}
\\
&\leq D_{jj} \sum_{\ell} M_{j\ell} c_{\ell} s_{\ell} \ \ \ \mbox{by
the definition of $D_{jj}$}
\\
&< c_j \ \ \  \mbox{by the DMS condition.}
\end{align*}
which satisfies the desired condition ($\starcond$) of Lemma
\ref{pro:recM}. This completes the proof of Theorem~\ref{cor:dms}.
\end{proof}

\section{Application to $\integers^2$ in the hard-core model}\label{sec:grid}

In this section, we show how to apply Theorem \ref{cor:dms} and
Theorem \ref{thm:implications} to the two-dimensional integer
lattice $\integers^2$ and improve the lower bound on
$\lambda_c(\integers^2)$, resulting in the following theorem.

\begin{theorem}
\label{thm:z2} There exists a $t\times t$ matrix $\bM$ such that
$\Tsaw(\integers^2)\in\FM$ and the DMS Condition holds for
$\lambda^*=\bound$.

Therefore, the following hold for $\integers^2$ for all $\lambda \le
\lambda^*$:
\begin{enumerate}
\item
SSM holds on $\integers^2$.
\item
There is a unique infinite-volume Gibbs measure on $\integers^2$.
\item
If $\bM$ has maximum degree $\Delta$, if $t=O(1)$ and $\Delta=O(1)$,
then for every finite subgraph $G$ of $\integers^2$, Weitz's
algorithm \cite{Weitz} gives an FPAS for approximating the partition
function $Z(G)$.
\item
For every finite subgraph $G$ of $\integers^2$, the Glauber dynamics
has $O(n\log n)$ mixing time.
\end{enumerate}
\end{theorem}

We first illustrate our approach by showing that Theorem
\ref{thm:z2} holds with $\lambda^*=1.8801$ for a simple choice of
$\bM$. We then explain how to extend the approach to higher
$\lambda$.

The graph $\integers^2$ is translation-invariant, hence the tree
$\Tsaw(\integers^2,v)$ is identical for every vertex
$v\in\integers^2$.  Fix a vertex, call it the origin $\orig$, and
let us consider $\Tsaw(\integers^2)=\Tsaw(\integers^2,\o)$. Each
path from the root of $\Tsaw(\integers^2)$ corresponds to a
self-avoiding walk in $\integers^2$ starting at the origin. Any walk
on $\integers^2$ starting at the origin $\orig$ can be encoded as a
string over the alphabet $\{N, E, S, W\}$ corresponding to North,
East, South and West. The tree $\Tsaw(\integers^2)$ contains such
strings, truncated the first time the corresponding walk completes a
cycle. A relaxed notion of such a tree would be to truncate a walk
only when a 4-cycle is completed. Denote such a tree by $\Tfour$,
and clearly we have that $\Tsaw(\integers^2)$ is a subtree of
$\Tfour$. Our first idea is to define a branching matrix $\bN$ so
that $\Tfour \in \cF_{\le \bN}$, and hence $\Tsaw(\integers^2) \in
\cF_{\le \bN}$.

To avoid cycles of length four, it is enough to track the last three
steps of the walks. Labeling the paths using $\{N,E,S,W\}$ as
mentioned above, their branching rule is easily determined. For
example, a path labeled $NWW$ is followed by paths labeled $WWS$,
$WWN$ and $WWW$ which corresponds to adding the directions $S$, $N$
and $W$ respectively. As another example, a path labeled $NWS$ is
followed by paths labeled $WSW$ and $WSS$ corresponding to adding
the directions $W$ and $S$ to the path, while adding the direction
$E$ would have resulted in  a cycle of length 4. The number of types
in the corresponding branching matrix is $\le 4+4^2+4^3\leq 5^3$.
Indeed, we can reduce the representation of such paths by using
isomorphisms between the generating rules among them. This results
in 4 types in the following branching matrix $\bN$:
\begin{equation}\label{eq:m1}
\bN=\begin{pmatrix}
0&4&0&0\\
0&1&2&0\\
0&1&1&1\\
0&1&1&0
\end{pmatrix},
\end{equation}
{where the type $i=0,...,3$ of a vertex (walk) in the tree
represents the fact that a continuation with a minimum of $4-i$
additional edges are needed to complete a cycle of length 4.}

\begin{figure}[h]
\begin{center}
\includegraphics[width=12cm,angle=0]{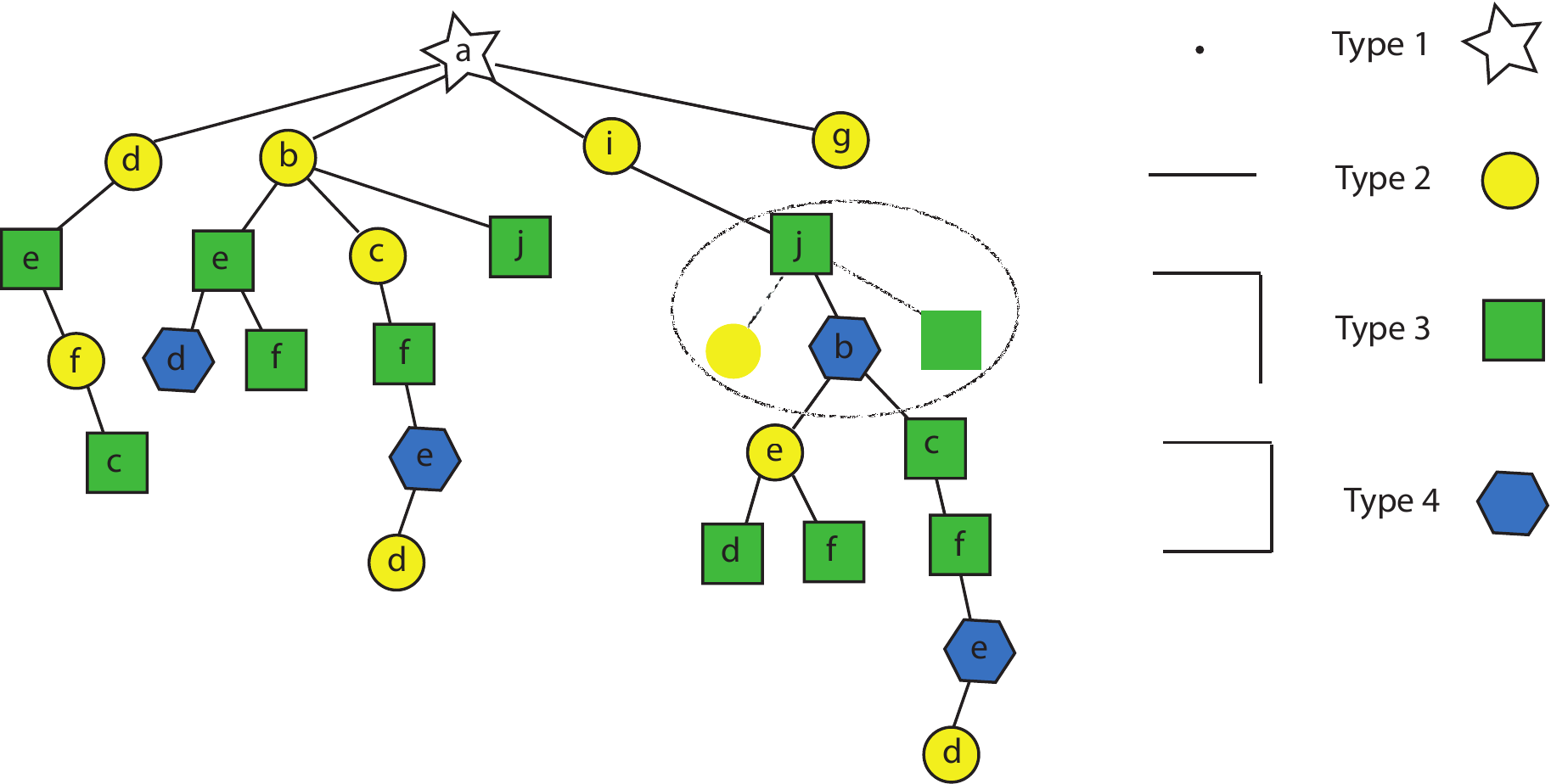}
\caption{Assignment of the four types from matrix $\bN$ defined in
\eqref{eq:m1}
 to the self-avoiding walk tree $\Tsaw$ from Figure \ref{fig:saw}. In the circled area, we
also draw redundant leaves at vertex $j$ which may appear in the
branching rule, but not in $\Tsaw$.} \label{fig-4types}
\end{center}
\end{figure}

See Figure \ref{fig-4types} for an illustration of this branching
matrix $\bN$. One can verify that this branching matrix captures,
inter alia,
 the self-avoiding walk trees from $\integers^2$:
\begin{observation}
\label{lem:approx1} For any finite subgraph $G=(V,E)$ of
$\integers^2$ and $v\in V$, $\Tsaw(G,v)\in \mathcal F_{\leq \bN}$.
\end{observation}

For this branching matrix, one can check that the {(DMS)} condition
of Theorem \ref{cor:dms} holds with $\lambda^*=1.8801$,
$\bS=\textrm{\textbf{Diag}} ( 1.040, 1.388, 1.353, 1.255) $ and
$\bc=( 0.266037, 0.100891, 0.100115, 0.0973861 ) $. Checking the DMS
Condition for a given choice of parameters would have been a
straightforward task, were it not for the irrationality of the
coefficients $D_{jj}$. However, one can establish rigorous upper
bounds for $D_{jj}$, based on concavity of the function (of
$\alpha$) used in the definition of $D_{jj}$, in a suitable range of
the parameters. These details will be discussed further below. As a
consequence, we can conclude that Theorem \ref{thm:z2} holds for
$\bN$ and $\lambda^* = 1.8801$.

The primary reason why the branching matrix $\bN$ improves beyond
the tree-threshold of $\lambda<\lambda_c(\T_4)=27/16=1.6875$ is that
the average branching factor of any $T\in \mathcal F_{\leq \bN}$ is
significantly smaller than that of the regular tree of degree $4$.

To obtain a further reduction in the average branching, we observe
that  $\bN$ did not consider the effect of occupying  (or
unoccupying) certain leaves as prescribed in Weitz's construction.
Starting with $\Tfour$,  prune the leaves as is done in
 the construction of $\Tsaw(\integers^2)$ from Section \ref{sec:saw}.
  Denote the new tree as $\Tfour'$.  Clearly we still have that $\Tsaw(\integers^2)$
 is a subtree of $\Tfour'$.

Let us illustrate the difference between $\Tfour$ and the pruned
tree $\Tfour'$. We first fix an underlying order for the neighbors
of each vertex.  To this end, say $N>E>S>W$ and this prescribes an
ordering of the neighbors of each vertex. Consider a leaf vertex
$v'$ in the tree $\Tfour$ corresponding to the vertex $v$ in
$\integers^2$ and to the path $\rho$ in $\integers^2$. Since $v'$ is
a leaf vertex in $\Tfour$, $\rho$ must end with a cycle at $v$, say
$WNES$.
  Since $v$ was exited in the West direction at the beginning of the 4-cycle, and since $W<N$,  the leaf vertex $v'$ would be labeled occupied in Weitz's construction, thus resulting in the removal of $v'$ and its parent in
  the construction of $\Tfour'$.
  Note, every vertex $w'$ in $\Tfour$ of type $WNE$ has a child $v'$ of type
  $NES$, and consequently $w'$ (and its subtree) will be removed from
  the tree in the pruning process to construct $\Tfour'$.  Thus, after removing vertices of type $WNE$ (and similarly, $WSE$, $SEN$ and $ENW$) from $\Tfour$, it is still the case that $\Tsaw(\integers^2)$ is a subtree of the resulting tree ($\Tfour'$).
This highlights why $\Tfour'$ has a significantly smaller average
branching factor than $\Tfour$.

We can define a branching matrix $\bM_2$, with 17 types (as
illustrated in Figure \ref{fig-17types}), such that $\Tfour'\in
\mathcal F_{\le \bM_2}$, and hence $\Tsaw(\integers^2)\in \mathcal
F_{\le \bM_2}$. We can prove the DMS Condition is satisfied for
$\bM_2$ at $\lambda^*=\boundF$, as we will describe shortly, which
significantly improves upon our initial bound resulting from
considering $\Tfour$.
\begin{figure}[h]
\begin{center}
\includegraphics[width=10cm,angle=0]{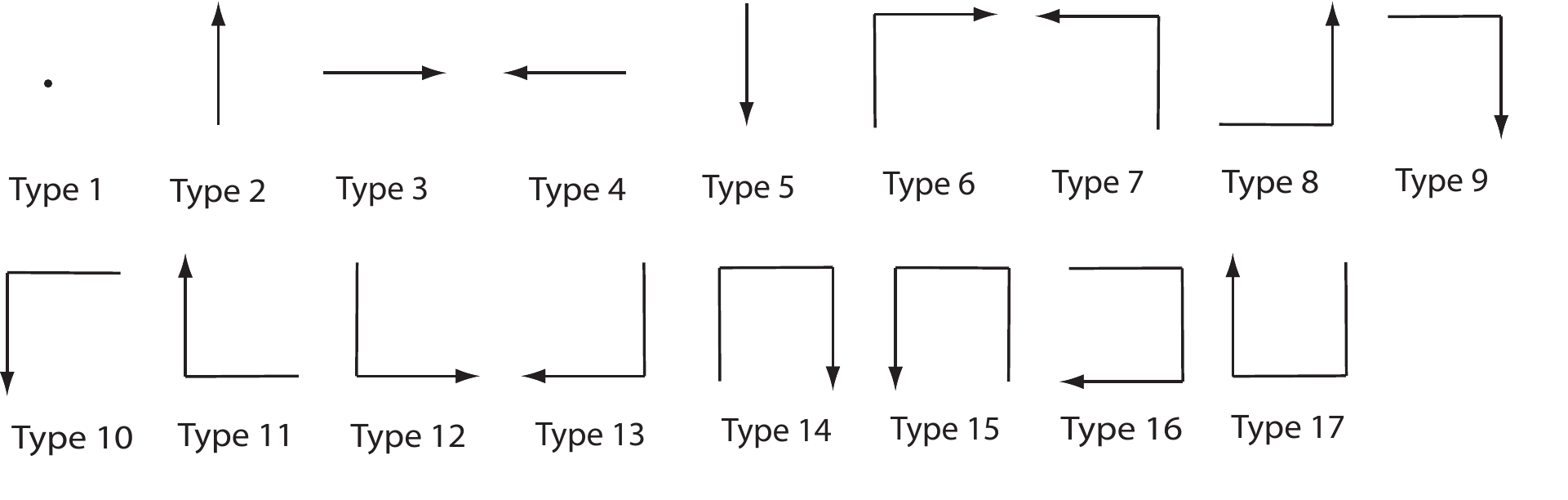}
\caption{Shapes that the seventeen types (or labels) represent for
$\bM_2$ where $\Tfour'\in \mathcal F_{\le \bM_2}$.}
\label{fig-17types}
\end{center}
\end{figure}

A natural direction for improved results is to consider branching
matrices corresponding to avoidance of larger cycles, while also
accounting for the removal of vertices prescribed by the
construction of Weitz. We briefly outline  such an approach for
walks avoiding cycles of length at most 4, 6, and 8, respectively.
Avoiding cycles of length $2i$ results in $\sum_{j\leq 2i-1} 4^j
\leq 5^{2i-1}$ types, hence the computations become increasingly
difficult for larger $i$. For 8-cycles the task of  finding
{appropriate} parameters to satisfy the conditions of Theorem
\ref{cor:dms} is still feasible.

More precisely, we can define branching matrices $\bM_i$ for $i \ge
2$, that (i) represent the structure of trees of walks avoiding
cycles of length $\leq 2i$, as well as (ii) account for the removal
of vertices based on children being labeled `occupied.' One can
extend the above construction of $\bM_2$ for general $i > 2$ by
using types encoded by longer paths with length at most $2i$ and
ruling out the types that either contain a cycle of length at most
$2i$ or whose children end up being labeled occupied. We can make
the following general observation from our construction.
\begin{observation}
\label{lem:approx2} For any finite subgraph $G=(V,E)$ of
$\integers^2$ and $v\in V$, $\Tsaw(G,v)\in \mathcal F_{\leq \bM_i}$
for any $i \ge 2$.
\end{observation}

As mentioned earlier, the matrix $\bM_2$ constructed above consists
of $17$ types. An explicit description of it is shown in the Online
Appendix \cite{URL}, along with the associated parameters $\bS$ and
$\bc$ for which one can check the DMS Condition for $\lambda^* =
\boundF$; this establishes  Theorem \ref{thm:z2} for $\bM_2$ and
$\lambda^* = \boundF$.

The following table summarizes the threshold $\lambda^*$ we obtain
for each $\bM_i$:
\begin{center}
\begin{tabular}{|c|c|c|c|}
\hline
Max length of Avoiding-cycles&Effect of Occupations&Number of Types &$\lambda^*$\\
\hline\hline
4 &No&4&1.8801\\
\hline
4&Yes&17  $(<5^3)$&\boundF\\
\hline
6&Yes&132 $(<5^5)$&2.3335\\
\hline
8&Yes&922 $(<5^7)$&\bound\\
\hline
\end{tabular}
\end{center}

Note that, one can further improve the bound on $\lambda$ by using
more types for higher $i$ and hence Theorem \ref{thm:z2} on
$\integers^2$ will hold with the corresponding activity $\lambda^*$.
For any such matrix, the verification of the DMS Condition relies on
(i) `guessing' appropriate values for the parameters $\bS$ and $\bc$
and (ii) formally verifying that DMS Condition holds for the chosen
$\bS$ and $\bc$. In choosing desirable  $\bS$ and $\bc$, we employed
a heuristic random walk algorithm.

To verify that the DMS Condition holds for a given {\em rational}
matrix $\bS$ and vector $\bc$ is straightforward, provided we can
obtain a rational upper bound for each type $j$ for the function:
\[
f_j(\alpha) = \frac{\left( 1-\alpha\right) \left(
1-\theta_{j}\left(
\frac{1-\alpha}{\lambda\alpha}\right)^{1/\Delta_{j} }\right)
}{s_{j}-\alpha}\,.
\]
Indeed, due to the concavity of this function for $0 < \theta_j \le
1$, $s_j > 51/50$ and $\lambda > 27/16$, \footnote{This is a
nontrivial algebraic fact. It can be proved by transforming the
second derivatives condition to a set of integer polynomial
constraints and using the ``\textbf{resolve}'' function in
MATHEMATICA for the satisfiability of the constraints, which is
rigorous by the Tarski-Seidenberg Theorem \cite{TARSKI} for the real
polynomial systems \cite{mathematica} and the so-called cylindrical
algebraic decomposition \cite{CAD}.} it is always possible to find a
\emph{provable} upper bound for $f_j$ in such a regime. This can be
done, for example,  by describing a suitable `envelope' for $f_j$
consisting of a piecewise linear function of the form:
\[
g_j\left(  \alpha\right)  =\left\{
\begin{array}
[c]{cc}
B_{\ell} & \text{if }\alpha<\alpha_{\ell}\\
\min\{b_{\ell}\left(  \alpha-\alpha_{\ell}\right)
+B_{\ell},b_{u}\left(  \alpha
-\alpha_{u}\right)  +B_{u}\} & \text{if }\alpha_{\ell}<\alpha<\alpha_{u}\\
B_{u} & \text{if }\alpha>\alpha_{u}
\end{array}
\right.
\]
where $\alpha_{\ell},\alpha_{u}$ are points such that
$b_{\ell}>f_{j}^{\prime }\left(  \alpha_{\ell}\right)  >0$,
$b_{u}<f_{j}^{\prime}\left(  \alpha _{u}\right)  <0$,
$B_{\ell}>f_j\left(  \alpha_{\ell}\right)  $ and $B_{u}>f_j\left(
\alpha_{u}\right)  $.  It is clear for any such function that
$g_j(\alpha)>f_j(\alpha)$, thus we obtain a provable upper bound for
$f_j$ using $g_j$.

For every $\bM_i$ in the above table, we provide $\bS$ and $\bc$,
along with appropriate envelopes that lead to upper bounds
$\hat{D}_{jj}$ for the corresponding $D_{jj}$.  Then we verify that
the DMS Condition holds for the given values of $\lambda$ by
replacing $D_{jj}$ with $\hat{D}_{jj}$. For $i=2,3,4$ these values
($\bM$, $\bS$, $\bc$, $\alpha_{\ell}$ and $\alpha_{u}$) are given in
the Online Appendix \cite{URL}.

\section{Ising Model}
\label{sec:Ising} The approach taken here for  the hard-core model
can also be employed to address corresponding questions in the
well-studied Ising model. The Ising model, with inverse temperature
parameter $\beta \ge 0$, on a finite graph $G=(V,E)$ is the model
associated with the Gibbs distribution $\mu$ on
$\Omega=\{-1,+1\}^{|V|}$ such that for $\sigma=[\sigma_i]\in\Omega$,
$$\mu(\sigma)~=~\frac1Z\exp\Bigl(\beta\sum_{(i,j)\in
E}\sigma_i\sigma_j\Bigr)\,,$$ where the normalizing constant is the
partition function: $Z=Z(G, \beta):=\sum_{\sigma \in\Omega}
\exp\left(\beta\sum_{(i,j)\in E}\sigma_i\sigma_j\right)$. The
notions of SSM, the self-avoiding walk tree representation, and
branching trees defined for the hard-core model extend identically
to the Ising model (or, for that matter, any other $2$-spin model).
Moreover, an analog of Lemma \ref{pro:recM} also follows easily for
Ising. Then, by the use of an appropriate statistic $\varphi$, the
following simpler analog of the DMS Condition can be proved for the
Ising model.

\begin{theorem}\label{thm:Ising}
Given a $t\times t$ branching matrix $M$ and $\beta^*>0$, suppose
there exists $\bc=(c_1,\ldots,c_t)> 0$ such that
\begin{equation}
\label{eq:DMS_Ising} \tanh(\beta^*)\bM\bc < \bc,
\end{equation}
then SSM and the conclusions of Theorem~\ref{thm:implications} hold
for $\bM$ and all $\beta\in[0,\beta^*]$.
\end{theorem}

\begin{proof}
First we note that Theorem \ref{thm:saw} holds in general for all
two spin models including the Ising model. Hence, Corollary
\ref{cor:weitz} and Remark \ref{rem:single-type} are applicable to
the Ising model as well. Further, observe that the proof of Theorem
\ref{thm:implications} (i.e., the implications of SSM) still hold
for the Ising model. Consequently, we can prove Theorem
\ref{thm:Ising} using similar notation and proof approach as was
used for Theorem~\ref{cor:dms}.

Given a tree $T\in \mathcal F_{\leq \bM}$ and configuration
$\bfrho$, let us define again $\alpha=\alpha^{\bfrho}_r(T,\beta)$ as
the probability that the root $r$ of $T$ is minus-spinned. (Recall
that in the hard-core model this was the probability that $r$ was
unoccupied.) If $w_{1},\ldots,w_{k}$ are the children of $r$ and
$T_{1},\ldots,T_{k}$ are the corresponding subtrees subtended at
them, we let $\alpha_{i}:= \alpha^{\bfrho}_{w_i}(T_i,\beta)$ for
$i\leq k$. For $i>k$, we define $\alpha_{i}:=1/2$.  Further let
$\vartheta_i=\frac{1-\alpha_i}{\alpha_i}$,
 and $\vartheta=\frac{1-\alpha}{\alpha}$.
Using these notations, a straightforward recursion calculation with
the partition function leads to the following:
\begin{equation}
\vartheta = \prod_{i=1}^{\Delta_j}
\frac{\exp(2\beta)\vartheta_i+1}{\vartheta_i+\exp(2\beta)}\,,
\label{eq:one-step1-Ising}
\end{equation}
where $j$ is the type of $r$, and $\Delta_j=\sum_{\ell} M_{j\ell}$.

Motivated by \eqref{eq:one-step1-Ising}, the function $F_j$ (defined
in Section \ref{sec:reg} for the hard-core model) can be redefined
for the Ising model as follows.
$$F_j(m_1,\ldots,m_{\Delta_j}):=\varphi_j\left(
\prod_{i=1}^{\Delta_j}
\frac{\exp(2\beta)\varphi^{-1}_{j_i}(m_i)+1}{\varphi^{-1}_{j_i}(m_i)+\exp(2\beta)}\right),$$
where $j_i$ is the type of child $w_i$ and  $\varphi_j$ is the
statistic
 for a vertex of type $j$. Further, we define
$m:=\varphi_j(\vartheta)$ and $m_i:=\varphi_{j_i}(\vartheta_i)$. It
follows from \eqref{eq:one-step1-Ising} that
$m=F_j(m_1,\ldots,m_{\Delta_j}).$ Then, one can prove the `Ising
version' of Lemma \ref{pro:recM} with the interval
$I=[\exp(-2\beta\Delta),\exp(2\beta\Delta)]$ using the same
arguments as those in the proof of Lemma \ref{pro:recM}. Further,
using the same arguments as in the proof of Theorem \ref{cor:dms},
with the choice of $\varphi_j(x):=\log (x)$, we have that
\[
\frac{\partial F_j}{\partial m_{i}}=\frac{\vartheta_{i}\left(
e^{4\beta }-1\right)  }{\left(  e^{2\beta}\vartheta_{i}+1\right)
\left(  e^{2\beta }+\vartheta_{i}\right)  }\leq\tanh\left(
\beta\right),
\]
from which the desired condition \eqref{eq:DMS_Ising} follows
easily.
 This completes the proof of Theorem
\ref{thm:Ising}.

\end{proof}
 Using Theorem \ref{thm:Ising} with branching matrices $\bM$
analogous to those we employed in Section~\ref{sec:grid} for the
hard-core model, we can prove that SSM holds for the Ising model on
$\integers^d$ for all $\beta<\beta^*$ as detailed in the following
table:

\begin{center}
\begin{tabular}{|c|c|}
\hline
Dimension&$\beta^*$\\
\hline\hline
2 & 0.392190\\
\hline
3& 0.214247\\
\hline
4& 0.148045\\
\hline
5& 0.113347\\
\hline
\end{tabular}
\end{center}
In comparison, applying Weitz's general technique to $\integers^2$
implies SSM for $\beta<.34657$.

We do not investigate the Ising model further because there are much
stronger results known for this model. Onsager \cite{Ons44}
established that $\beta_c(\integers^2)= \log(1+\sqrt{2}) \approx
0.440686$. And for general trees, Lyons \cite[Theorem 2.1]{L89}
established the critical point for uniqueness.

\section{Acknowledgements}
We are grateful to Karsten Schwan and the CERCS group at Georgia
Tech for lending the support of their GPU machines so that we could
conduct efficient parallel computations.




\end{document}